\documentclass[preprint]{imsart}

\RequirePackage[OT1]{fontenc}
\RequirePackage{amsthm,amsmath}
\RequirePackage[numbers]{natbib}
\RequirePackage[colorlinks,citecolor=blue,urlcolor=blue]{hyperref}

\startlocaldefs
\numberwithin{equation}{section}
\theoremstyle{plain}

\newtheorem{thm}{Theorem}[section]
\newtheorem{lem}{Lemma}[section]

\newtheorem{rem}{Remark}[section]
\newtheorem{cor}{Corollary}[section]

\newtheorem{ex}{Example}[section]

\endlocaldefs

\begin{document}

\begin{frontmatter}
\title{Is an arbitrary diffused Borel probability measure in a Polish space without isolated points Haar measure?
}
\runtitle{Is an arbitrary diffused Borel probability measure $\cdots$ Haar measure?
}

\begin{aug}
\author{\fnms{Gogi} \snm{Pantsulaia}\thanksref{t1}\ead[label=e1]{g.pantsulaia@gtu.ge}}
\affiliation{Georgian Technical University\\
 I.Vekua Institute of Applied Mathematics }

\thankstext{t1}{This paper was partially supported by Shota
Rustaveli National Science Foundation's Grant no FR/503/1-30/14.}
\runauthor{G.Pantsulaia}

\address{Department of Mathematics, Georgian Technical University,\\ Kostava Street. 77 , Tbilisi
DC  0175, Georgian Republic \\
\printead{e1}}
\end{aug}

\begin{abstract}It is introduced a certain approach for equipment of  an arbitrary  set of the cardinality of the continuum by structures of Polish groups  and two-sided (left or right) invariant Haar measures. By using this approach we answer positively  Maleki's certain question(2012) {\it what are the real $k$-dimensional manifolds with at least two different Lie group structures that have the same Haar measure.} It is demonstrated that for each diffused Borel probability measure $\mu$ defined in
a Polish space $(G,\rho,\mathcal{B}_{\rho}(G))$ without isolated points there
exist a metric $\rho_1$ and a group operation $\odot$  in $G$ such
that   $\mathcal{B}_{\rho}(G)=\mathcal{B}_{\rho_1}(G)$  and   $(G,\rho_1, \mathcal{B}_{\rho_1}(G), \odot)$ stands a compact
Polish group with a two-sided (left or right) invariant Haar measure $\mu$, where $\mathcal{B}_{\rho}(G)$ and $\mathcal{B}_{\rho_1}(G)$
denote Borel $\sigma$ algebras of subsets of $G$ generated by metrics $\rho$ and $\rho_1$, respectively.
Similar result is obtained for construction of  locally compact non-compact or  non-locally compact Polish groups equipped with two-sided (left or right) invariant quasi-finite Borel  measures.

\end{abstract}



\end{frontmatter}

\section{Inroduction}

\vskip3mm

  Let $(G,\rho, \odot)$ be a  Polish group, by which we mean a group
with a complete separable metric $\rho$ for which the transformation (from $G
\times G$ onto $G$ )  sending  $(x,y)$ into $x^{-1} \odot y$ is
continuous.

Let $B_{\rho}(G)$ be  the $\sigma$-algebra of Borel subsets of $G$ defined by the metric $\rho$.

\vskip3mm

The following problem was under intensive consideration by many mathematicians exactly one century ago.

{\bf Problem 1.1.}  Let  $(G,\rho, \odot)$ be a locally compact Polish group which is dense-in-itself \footnote{ A subset $A$ of a topological space is said to be dense-in-itself if $A$ contains no isolated points.}, that is, a space homeomorphic to a separable complete metric space and  $G$ has no isolated points. Does there exist a Borel measure $\mu$ in $(G,\rho, \odot)$  satisfying the following properties:

(i) The measure $\mu$ is diffused, that is,  $\mu$ vanishes on all singletons;

(ii) The measure $\mu$ is a two-sided (left or right) invariant, that is,  $\mu( g_1 \odot E \odot g_2 ) =\mu(E)$( $\mu(g_1 \odot E) =\mu(E)$ or $\mu( E \odot g_2 ) =\mu(E)$ )  for every $ g_1,g_2 \in G$  and every Borel
set $E \in B(G)$;

(iii) The measure $\mu$ is outer regular, that is,
$$
(\forall E)(E \in B(G) \rightarrow \mu(E) = \inf\{\mu(U) : E \subseteq U ~\& ~U \mbox{is~open}\});
$$

(iv) The measure $\mu$ is inner regular, that is,

$$
(\forall E)(E \in B(G) \rightarrow \mu(E) = \sup\{\mu(F) : F \subseteq E ~\&~ F ~\mbox{is~compact}\});
$$

(v) The measure $\mu$  is finite on every compact set, that is $\mu(K) < \infty$  for all compact $K$.

The special case of a left (or right ) invariant measure for second countable \footnote{A topological space $T$ is second countable if there exists some countable collection $\mathcal{U} = \{U_i\}_{i \in N}$ of open subsets of $T$ such that any open subset of $T$ can be written as a union of elements of some subfamily of $\mathcal{U}$} locally compact groups had been shown by Haar in 1933 \cite{Haar 1933}.
Notice that each Polish space is second countable which implies that the answer to Problem 1.1 is yes. The measure $\mu$ satisfying conditions (i)-(v) is called a left (right or two-sided) invariant Haar measure in a locally compact Polish group  $(G,\rho, \odot)$.

In this note  we would like to study the following problems, which can be considered as
 converse (in some sense) to Problem 1.1.

{\bf Problem 1.2.}  Let $(G,\rho)$ be a Polish metric space which is dense-in-itself.  Let $\mu$ be a diffused Borel probability measure defined in $(G,\rho)$. Do there exist a metric $\rho_1$ and a
group operation $\odot$ in $G$ such that the following three conditions

(j) The class of Borel measurable subsets of $G$ generated by the metric $\rho_1$ coincides with the class of Borel
measurable subsets of the same space generated by the metric $\rho$,

(jj) $(G,\rho_1, \odot)$ is  a compact Polish group

and

(jjj) $\mu$ is a left(right or two-sided) invariant Haar  measure in $(G,\rho_1, \odot)$

hold true ?

{\bf Problem 1.3.} Let $(G,\rho)$ be a Polish metric space which is dense-in-itself. Let $\mu$ be a diffused $\sigma$-finite non-finite Borel measure defined in $(G,\rho)$. Do there  exist a metric $\rho_{\varphi}$, a group operation $\odot_{\varphi}$ in $G$ and the Borel measure $\mu^\star$ in $G$ such that the following  four conditions

(i) The class of Borel measurable subsets of $G$ generated by the metric $\rho_{\varphi}$ coincides with the class of Borel
measurable subsets of the same space generated by the metric $\rho$,

(ii) $(G,\rho_{\varphi}, \odot_{\varphi})$  is  a non-compact locally compact Polish group,

(iii) The measures $\mu^\star$ and $\mu$ are equivalent

and

(iv) $\mu^\star$ is a left (right or two-sided) invariant  $\sigma$-finite non-finite  Haar  measure in $(G,\rho_{\varphi}, \odot_{\varphi})$

hold true?

{\bf Problem 1.4.}  Let $(G,\rho)$ be a Polish metric space which is dense-in-itself. 
Let $\mu$ be a diffused non-$\sigma$-finite  quasi-finite Borel measure defined in $(G,\rho)$. 
Do there exist a metric $\rho_1$ and a
group operation $\odot$ in $G$ such that the following three conditions

(j) The class of Borel measurable subsets of $G$ generated by the metric $\rho_1$ coincides with the class of Borel
measurable subsets of the same space generated by the metric $\rho$,

(jj) $(G,\rho_1, \odot)$ is a non-locally compact  Polish group

and

(jjj) $\mu$ is a left(right or two-sided) invariant quasi-finite Borel measure in $(G,\rho_1, \odot)$

hold true ?

In \cite{Maleki2012}, the author uses methods of the theory ultrafilters   to present a modified proof that a locally compact
group with a countable basis has a left invariant and right invariant Haar measure. The
author first shows that the topological space $(\beta_1 X; \tau_1)$  consisting of all ultrafilters on
a non-empty set $X$ is homeomorphic to the topological space $(\beta_2 X; \tau_2)$ of all nonzero
multiplicative functions in the first  dual space $\ell_{\infty}^{*}(X)$ (Theorem 3.8). By using this
result the author proves the existence of the infinitely additive left invariant measure $\lambda$
on compact sets of the locally compact Hausdorf topological group $G$ (Theorem 7.1).
Starting from this point, the author introduces the notion of $\nu$-measurable subsets
in $G$ where $\nu$ is an outer measure in $G$ induced by the $\lambda$ and open sets in $G$, and
proves the existence of a left invariant Haar measure by the scheme presented in  \cite{Neumann99}. Notice that his proof essentially uses the axiom of choice.
Several examples of the Haar measure are presented. It is underlined by Example 9.7
that $G = R^k$ with $k = \frac{n^2-n}{2}$
has two Lie group structures but the Lebesgue measure on
$R^k$ is the Haar measure on both Lie groups. In this context the following question was stated in this paper.

\medskip

{\bf Problem 1.5}(\cite{Maleki2012},Question 9.8) What are the real $k$-dimensional manifolds with at least two different
Lie group structures that have the same Haar measure?

\medskip

The rest of the paper is the following.

In Section 2 we  introduce a certain approach for equipment of an arbitrary  set of the cardinality of  the continuum
by structures of various(compact, locally compact or non-locally compact) Polish groups with  two-sided(left or right) invariant Borel measures and
study Problem  1.5.  

In Section 3 we study general question whether  an arbitrary diffused Borel probability 
measure in a Polish space without isolated points is Haar measure and give its affirmative resolution.
Moreover, we study Problems 1.2, 1.3  and answer to them positively. 

\section{Equipment of an arbitrary  set of the cardinality of the continuum by structures of Polish groups }

\begin{thm}Let $X$ be a set of the cardinality of the continuum and $(G, \odot, \rho)$  a Polish group.
Further, let $f : G \to X$ be a one-to-one mapping. We set
$$
x \odot_f y = f(f^{-1}(x)\odot f^{-1}(y))
$$
and
$$
\rho_f(x,y)=\rho(f^{-1}(x),f^{-1}(y))
$$
for $x,y \in X$. Then  the following  conditions hold true:

(i)~ $(G_f,\odot_f, \rho_f)$  is a Polish group which is Borel isomorphic to the Polish group $(G, \odot, \rho)$;

(ii)~ If $(G, \odot, \rho)$   is an abelian  Polish group then so is $(G_f,\odot_f, \rho_f)$;

(iii)~ If  $\rho$ is two-sided invariant metric in $(G, \odot)$ so is $\rho_f$  in $(G_f,\odot_f)$;

(iv)~ If $(G, \odot, \rho)$ is dense-in-itself  so is $(X,\odot_f, \rho_f)$;

(v)~ If  $(G, \odot, \rho)$  is a compact Polish group  then  so is  $(X,\odot_f, \rho_f)$;

(vi)~ If  $(G, \odot, \rho)$  is a locally compact Polish group  then  so  is  $(X,\odot_f, \rho_f)$;

(vii)~ If  $(G, \odot, \rho)$  a non-locally compact Polish group  then  so is $(X,\odot_f, \rho_f)$;

(viii) ~ If  $(G, \odot, \rho)$  is a locally compact or compact Polish group and $\lambda$ is a left(or right or two-sided ) invariant Haar  measure in $(G, \odot, \rho)$, then $\lambda_f$ also is a left(or right or two-sided ) invariant Haar measure in  $(G_f,\odot_f, \rho_f)$, where $G_f=X$, $\mathcal{B}_{\rho_f}(G_f)$ is Borel $\sigma$-algebra
of $G_f$ generated by the metric $\rho_f$  and $\lambda_{f}$ is a Borel measure in $G_f$ defined by
$$
(\forall Y)(Y \in \mathcal{B}_f(G_f) \rightarrow
\lambda_{f}(Y)=\lambda(f^{-1}(Y))).
$$

(ix) ~ If  $(G, \odot, \rho)$  is a non-locally compact Polish group and $\lambda$ is a left(or right or two-sided)  invariant quasi-finite \footnote{A measure $\mu$ is called quasi-finite if there is a $\mu$-measurable set $X$ with $0<\mu(X)<+\infty$.} Borel measure in $(G, \odot, \rho)$,
then $\lambda_f$ also is a left(or right or two-sided) invariant quasi-finite Borel  measure in  $(G_f,\odot_f, \rho_f)$, where $G_f=X$, $\mathcal{B}_{\rho_f}(G_f)$ is Borel $\sigma$-algebra
of $G_f$ generated by the metric $\rho_f$  and $\lambda_{f}$ is a Borel measure in $G_f$ defined by
$$
(\forall Y)(Y \in \mathcal{B}_f(G_f) \rightarrow
\lambda_{f}(Y)=\lambda(f^{-1}(Y))).
$$
\end{thm}

\begin{proof}{\bf Proof of the item (i).}

{\bf Closure }. If $x,y \in X$ then $x \odot_f y=f(f^{-1}(x)\odot
f^{-1}(y)) \in X$.

{\bf Associativity }. For all  $x, y$ and $z$ in $X$, we have

$$(x \odot_f y)\odot_f z=f[f^{-1}(x \odot_f y)\odot
f^{-1}(z)]=f[f^{-1}(f(f^{-1}(x)\odot f^{-1}(y)))\odot f^{-1}(z)]=
$$
$$
f[(f^{-1}(x)\odot f^{-1}(y))\odot f^{-1}(z)]=f[f^{-1}(x)\odot
(f^{-1}(y)\odot f^{-1}(z))]= $$
$$f[f^{-1}(x)\odot f^{-1}(y
\odot_f z)]=x \odot_f (y \odot_f z).
$$

{\bf Identity element.}~Let $e$ be an identity element of $G$.
Setting $e_f:=f(e)\in X$, for $x \in X$ we have
$$
x \odot_f e_f=x \odot_f f(e)= f(f^{-1}(x)\odot f^{-1}(f(e)))=
f(f^{-1}(x)\odot e)=f(f^{-1}(x))=x
$$ and
$$
 e_f \odot_f x =f(e) \odot_f x= f(f^{-1}(f(e))\odot f^{-1}(x))=
f(e \odot f^{-1}(x))=f(f^{-1}(x))=x.
$$
The latter relations means that $e_f$ is the identity element of
$X$.

{\bf Inverse element.} If $a \in G$ then we denote its inverse
element by  $a^{-1}_G$. For $x \in X$ setting
$x^{-1}_X=f((f^{-1}(x))^{-1}_G)$, we have
$$
x \odot_f x^{-1}_X=f(f^{-1}(x)\odot
f^{-1}(x^{-1}_X))=f(f^{-1}(x)\odot f^{-1}(f((f^{-1}(x))^{-1}_G)))=
$$
$$
f(f^{-1}(x)\odot (f^{-1}(x))^{-1}_G)=f(e)=e_f
$$
and
$$
x^{-1}_X \odot_f x=f( f^{-1}(x^{-1}_X)\odot
f^{-1}(x))=f(f^{-1}(f((f^{-1}(x))^{-1}_G))\odot f^{-1}(x) )=
$$
$$
f((f^{-1}(x))^{-1}_G\odot f^{-1}(x))=f(e)=e_f.
$$
The latter relations means that $x^{-1}_X$ is an inverse element
of $x$.

{\bf Continuity of the operation $(x,y) \to x \odot_f y^{-1}_X$ when
$(a,b) \to a \odot b^{-1}_G$ is continuous.}

For all neighbourhood $U_X(x \odot_f y^{-1}_X,r)$ we have to
choose such neighbourhoods $U_X(x,r_1)$ and $U_X(y,r_2)$ of
elements $x$ and $y$ respectively that $(w_1 \odot_f (w_2)^{-1}_X)\in U_X(x
\odot_f y^{-1}_X,r)$ for  $w_1 \in U_X(x,r_1)$
and $w_2 \in U_X(y,r_2)$.

We have
$$
U_X(x \odot_f y^{-1}_X,r)=\{z : \rho_f(z, x \odot_f y^{-1}_X)<r\}=
\{z : \rho(f^{-1}(z), f^{-1}(x \odot_f y^{-1}_X))<r\}=$$
$$ \{z :
\rho(f^{-1}(z), f^{-1}(f(f^{-1}(x) \odot f^{-1}(y^{-1}_X))))<r\}=
$$
$$\{z : \rho(f^{-1}(z), f^{-1}(f(f^{-1}(x) \odot
f^{-1}(f((f^{-1}(y))^{-1}_G))))<r\}= $$
$$ \{z : \rho(f^{-1}(z),
(f^{-1}(x) \odot (f^{-1}(y))^{-1}_G))))<r\}.
$$
Since $(a,b) \to a \odot b^{-1}_G$ is continuous, for
$a=f^{-1}(x)$, $b=f^{-1}(y)$ and  $r>0$ we can choose such
neighbourhoods $U_G(f^{-1}(x),r_1)$ and $U_G(f^{-1}(y),r_2)$ of
elements $f^{-1}(x)$ and $f^{-1}(y)$ respectively that then
$(a_1 \odot (a_2)^{-1}_G)\in U_G(f^{-1}(x) \odot
(f^{-1}(y))^{-1}_G,r)$ for $a_1
\in U_G(f^{-1}(x),r_1)$ and $a_2 \in U_G(f^{-1}(y),r_2).$

It is obvious to check the validity of the following equalities
$$
U_X(x,r_1)=f(U_G(f^{-1}(x),r_1)),$$
$$U_X(y,r_2)=f(U_G(f^{-1}(y),r_2)),$$
$$
U_X(x \odot_f y^{-1}_X,r)=f(U_G(f^{-1}(x) \odot
(f^{-1}(y))^{-1}_G,r).$$

Notice that if $w_1 \in U_X(x,r_1)$ and $w_2 \in U_X(y,r_2)$ then
$(w_1 \odot_f (w_2)^{-1}_X)\in U_X(x \odot_f y^{-1}_X,r)$.
Indeed,$w_1 \in U_X(x,r_1)$ and $w_2 \in U_X(y,r_2)$ imply that
$f^{-1}(w_1) \in U_G(f^{-1}(x),r_1)$ and $f^{-1}(w_2) \in
U_G(f^{-1}(y),r_2)$ from which we deduce that $(w_1 \odot
(w_2)^{-1}_G)\in U_G(f^{-1}(x) \odot (f^{-1}(y))^{-1}_G,r)$.

{\bf Borel isomorphism of $(G,\odot, \rho)$ and $(G_f,\odot_f, \rho_f)$}. Notice that this isomorphism is realized by the mapping $f:G \to G_f$.

{\bf Proof of the item (ii).}

Since $(G, \odot)$ is an abelian Polish group, for $x,y \in G_f$ we have
$$
x \odot_f y=f(f^{-1}(x)\odot f^{-1}(y))=f(f^{-1}(y)\odot
f^{-1}(x))=y \odot_f x.
$$

{\bf Proof of the item (iii).}

Since $\rho$ is a two-sided invariant metric in $(G, \odot)$ we have $\rho(h_1 \odot x \odot h_2,h_1 \odot y \odot h_2)=\rho( x,y)$ for each
$x,y,h_1,h_2 \in G$. Take into account this fact and the associativity property of the group operation $\odot_f$,  we get that the condition

$$\rho_f(h^*_1 \odot_f x^* \odot_f h^*_2,h^*_1 \odot_f y^* \odot_f h^*_2)=
$$
$$
\rho_f( f (f^{-1}(h^*_1) \odot f^{-1}(x^*) \odot f^{-1}(h^*_2)), f (f^{-1}(h^*_1) \odot f^{-1}(y^*) \odot f^{-1}(h^*_2)))=
$$
$$
\rho( f^{-1}(h^*_1) \odot f^{-1}(x^*) \odot f^{-1}(h^*_2), f^{-1}(h^*_1) \odot f^{-1}(y^*) \odot f^{-1}(h^*_2))=
$$
$$
\rho(f^{-1}(x^*),  f^{-1}(y^*))=\rho_f (f(f^{-1}(x^*)),  f(f^{-1}(y^*)))=\rho_f (x^*,  y^*)
$$
holds true for each $x^*, y^*, h^*_1, h^*_2 \in G_f$.

{\bf Proof of the item (iv).}  We have to show that if $(G, \odot, \rho)$ is dense-in-itself then so is $(G_f, \odot_f, \rho_f)$.
Indeed assume the contrary and let $x^*$ be an isolated point of $G_f$. The latter relation means that for some $\epsilon>0$  we have
 $\rho_f(y^*,x^*)\ge \epsilon$ for each $y^* \in G_f \setminus \{ x^*\}$ which implies that  $\rho(y,x)\ge \epsilon$ for each
$y \in G \setminus \{ x\}$  where $x=f^{-1}(x^*)$. We get the contradiction and the validity of the item (iv) is proved.

 {\bf Proof of the item (v).} We have to prove that if a family of open sets  $(U^*_i)_{i \in I}$ whose union  covers the space $G_f$ then there is its subfamily whose union also covers the same space.  Let consider a family of sets
  $(f^{-1}(U^*_i))_{i \in I}$. Since it  is  the family of open sets whose union covers the space $G$ and $G$ is a compact space, we claim that there is a finite subfamily
 $(f^{-1}(U^*_{i_k}))_{1 \le k \le n}(i_k \in I$ for $k=1,\cdots, n)$) whose union  $\cup_{k=1}^n f^{-1}(U^*_{i_k})$ covers $G$.
  Now it is obvious that the family $(U^*_{i_k})_{1 \le k \le n}$  is the family of open sets (in $G_f)$)  whose union  also covers $G_f$.

 {\bf Proof of the item (vi).} Let $x^* \in G_f$. Since $(G, \odot, \rho)$ is locally compact the point $f^{-1}(x^*)$ has a compact neighbourhood $U$.
  Now it is obvious that the  set $f(U)$ will be a compact neighbourhood of the point $x^*$. Since $x^* \in G_f$ was taken arbitrary the validity of the item (vi) is proved.

{\bf Proof of the item (vii).} Since $(G, \odot, \rho)$ is no locally compact there is a point $x_0$ which has no a compact neighbourhood.
Now if we consider a point $f(x_0)$, we observe that it has no a compact neighbourhood. Indeed, if assume the contrary and  $U$ is a compact neighbourhood of the point $f(x_0)$  then  $f^{-1}(U)$ also will be a compact neighbourhood of the point  $x_0$ and we get the contradiction. This ends the proof of the item (vii).

{\bf Proof of the item (viii).}

{\bf Proof of the diffusivity of the measure $\lambda_f$.} Since $\lambda$ vanishes on all singletons, we have
$$\lambda_{f}(x)=\lambda(f^{-1}(x)))=0$$
for each $x \in G_f$;

{\bf Proof of the left(or right or two-sided ) invariance of the measure $\lambda_f$.} If  $(G, \odot, \rho)$  is a locally compact or compact Polish group and $\lambda$ is a left(or right or two-sided ) invariant Haar  measure in $(G, \odot, \rho)$,
then $\lambda_f$ also will be a left(or right or two-sided ) invariant Haar measure in  $(G_f,\odot_f, \rho_f, \mathcal{B}_{\rho_f}(G_f))$, where $G_f=X$, $\mathcal{B}_{\rho_f}(G_f)$ is Borel $\sigma$-algebra
of $G_f$ generated by the metric $\rho_f$  and $\lambda_{f}$ is defined by
$$
(\forall Y)(Y \in \mathcal{B}_{\rho_f}(G_f) \rightarrow
\lambda_{f}(Y)=\lambda(f^{-1}(Y))).
$$

{\bf  Case 1. } $\lambda$ is a left invariant Haar  measure in $(G, \odot, \rho)$.

$$
(\forall Y)(\forall h) ((Y \in \mathcal{B}_{\rho_f}(G_f) ~\&~ h \in G_f) \rightarrow
\lambda_{f}(h \odot_f Y )=$$
$$\lambda(f^{-1}( h \odot_f Y ))=\lambda(f^{-1}(h)  \odot  f^{-1}(Y))=\lambda(f^{-1}(Y))=\lambda_{f}(Y ))).
$$

{\bf  Case 2. } $\lambda$ is a right invariant Haar  measure in $(G, \odot, \rho)$.
$$
(\forall Y)(\forall h) ((Y \in \mathcal{B}_{\rho_f}(G_f) ~\&~ h \in G_f) \rightarrow
\lambda_{f}( Y  \odot_f h )=$$
$$\lambda(f^{-1}( Y \odot_f h ))=\lambda( f^{-1}(Y) \odot f^{-1}(h))=\lambda(f^{-1}(Y))=\lambda_{f}(Y ))).
$$

{\bf  Case 3. } $\lambda$ is a two-sided invariant Haar  measure in $(G, \odot, \rho)$.

$$
(\forall Y)(\forall h_1)(\forall h_2) ((Y \in \mathcal{B}_{\rho_f}(G_f) ~\&~ h_1 \in G_f) ~\&~ h_2 \in G_f) \rightarrow
\lambda_{f}(h_1 \odot_f  Y  \odot_f  h_2 )=$$
$$\lambda(f^{-1}(h_1 \odot_f  Y  \odot_f  h_2 ))=\lambda( f^{-1}(h_1) \odot f^{-1}(Y) \odot f^{-1}(h_2))=
\lambda(f^{-1}(Y))=\lambda_{f}(Y ))).
$$

{\bf Proof of the outer regularity of the measure  $\lambda_f$. } Let take any set $E_f \in B_{\rho_f}(G_f)$ and any $\epsilon>0$.  Let consider a set $f^{-1}(E_f) \in B(G)$. Since $\lambda$ is outer regular there is an open subset $U$ of $G$ such that
$f^{-1}(E_f) \subseteq U$ and $\lambda (U \setminus f^{-1}(E_f))<\epsilon.$ Then we get  $$\lambda_f (f(U) \setminus E_f)= \lambda(f^{-1}(f(U) \setminus E_f))= \lambda(U \setminus f^{-1}(E_f))<\epsilon.$$

{\bf Proof of the inner  regularity of the measure  $\lambda_f$. }  Let take any set $E_f \in B_{\rho_f}(G_f)$ and any $\epsilon>0$.  Let consider a set $f^{-1}(E_f) \in B(G)$. Since $\lambda$ is inner regular there is a compact subset $F$  of $G$ such that
$F \subseteq f^{-1}(E_f)$ and $\lambda (f^{-1}(E_f)\setminus F)<\epsilon$. Then we get  $$\lambda_f (E_f \setminus f(F))= \lambda(f^{-1}(E_f \setminus f(F))= \lambda( f^{-1}(E_f)\setminus F)<\epsilon.$$

{\bf Proof of the finiteness of the measure  $\lambda_f$ on  all compact subsets.} Let take any compact set $F \subseteq  G_f$. Since $f^{-1}(F)$ is compact in $G$ and the measure $\lambda$  is finite on every compact set we get $\lambda_f(F)=\lambda(f^{-1}(F)) < \infty$.

{\bf Proof of the item (ix).}  The proof of this item can be obtained by the scheme used in the proof of the item (viii).

\end{proof}

Below we consider some examples which employ the constructions described by  Theorem 2.1.

\begin{ex} Let
$f: {R} \to (-c,c)$ be  defined by
$f(y)=\frac{c(e^y-1)}{1+e^y}$ for $y \in ~{R}$, where $c>0$.
Then $f^{-1}:(-c,c) \to ~{R}$ is defined by $f^{-1}(x)=\ln(\frac{c+x}{c-x})$ for $x \in (-c,c)$.  For $x,y \in
(-c,c)$ we put
$$
x +_f y =
f(f^{-1}(x)+f^{-1}(y))=f(\ln(\frac{c+x}{c-x})+\ln(\frac{c+y}{c-y}))=
f(\ln(\frac{(c+x)(c+y)}{(c-x)(c-y)}))=
$$
$$
\frac{c(e^{\ln(\frac{(c+x)(c+y)}{(c-x)(c-y)})}-1)}{1+e^{\ln(\frac{(c+x)(c+y)}{(c-x)(c-y)})}}=
\frac{c(\frac{(c+x)(c+y)}{(c-x)(c-y)}-1)}{1+\frac{(c+x)(c+y)}{(c-x)(c-y)}}=
$$

$$
c\frac{(c+x)(c+y)-(c-x)(c-y)}{(c-x)(c-y)+(c+x)(c+y)}=c\frac{2cx+2cy}{2c^2+2xy}=
\frac{x+y}{1+\frac{xy}{c^2}}.
$$
Note that $\lambda_{f}$ defined by
$$
(\forall Y)(Y \in \mathcal{B}_{\rho_f}((-c,c)) \rightarrow
\lambda_{f}(Y)=\lambda(\{ \ln(\frac{c+y}{c-y}): y  \in
Y\})=\int_Y\frac{c^2}{c^2-t^2}dt)
$$
will be Haar measure in $(-c,c)$, where $\lambda$ denotes a linear
Lebesgue measure in $~{R}$.
\end{ex}

\begin{rem} Example 2.2 demonstrates that the Haar measure space $(G_c, \star, \rho_{G_c},  \nu)$
which comes from \cite{Maleki2012}(cf. Example 9.1, p.61) exactly
coincides with a Polish group $(~{R}_f, +_f, \rho_f,
\lambda_f )$ where $\rho$ is a usual metric in $~{R}$,
$\lambda$ is a linear Lebesgue measure in $~{R}$ and $f:
~{R} \to (-c,c)$ is a mapping defined by
$f(y)=\frac{c(e^y-1)}{1+e^y}$ for $y \in ~{R}$.

It is well known(see, \cite{Yakovenko04}, Eq. 35, p. 5)  that the
relativistic law of adding velocities has the following form
$$v =\frac{v_1+v_2}{1+\frac{v_1 v_2}{c^2}}$$
for $v_1,v_2 \in (-c,c)$, where $c$ denotes the speed of light.
This operation of adding exactly coincides with the operation
$+_f$ under which $(-c,c)$ stands a locally compact non-compact Polish group. Hence
the Haar measure $\lambda_f$ can be used in studding properties of
the inertial reference frame $O_0$ which moves relative to $O$
with velocity $v$ in along the $x$ axis (see, \cite{Yakovenko04},
p. 1).

\end{rem}

\begin{ex}Let $(~{R}, +, \rho)$ be a one-dimensional
Euclidian vector space and $\lambda$  a linear Lebesgue measure in
$~{R}$. Let $f: ~{R} \to (0,+\infty)$ be defined by
$f(x)=e^x$. We put
$$
x +_f y = exp\{\ln(x) + \ln(y)\}=exp\{\ln(xy)\}=xy
$$
and
$$
\rho_f(x,y)=|\ln(x)-\ln(y)|
$$
for $x,y \in (0,+\infty)$. We define $\lambda_{f}$  by
$$
(\forall Y)(Y \in \mathcal{B}((0,+\infty)) \rightarrow
\lambda_{f}(Y)=\lambda(\{ \ln(y) : y \in Y \} )).
$$
By Theorem 2.1  we know that $\lambda_{f}$ is  Haar measure in
$(0,+\infty)$.
 Since
$$
(\forall Z)(Z \in \mathcal{B}(~{R}) \rightarrow
\lambda(Z)=\int_Z d x)
$$
we deduce that
$$
(\forall Y)(Y \in \mathcal{B}((0,+\infty)) \rightarrow
\lambda_{f}(Y)=\lambda(\ln(Y))=\int_{\ln(Y)}dx=\int_Y\frac{dx}{x}).
$$

Note that  Haar measure space $(G, \cdot, \rho_G, \nu)$ constructed
in \cite{Maleki2012}(see p.54) coincides with Haar measure space
$(~{R}_f, +_f, \rho_f, \lambda_f )$.
\end{ex}

\begin{ex}Let $X=(-c,c)$ where $c>0$. We define $f : ~{R}
\to (-c,c)$ by $ f(x)=\frac{2 c arctg(x)}{\pi} $ for $x \in
~{R}$. Then $f^{-1}(w)=tg(\frac{\pi w}{2c})$ for $w \in
(-c,c)$. We have

$$
x \odot_f y = f(f^{-1}(x)\odot f^{-1}(y))=f(tg(\frac{\pi
x}{2c})+tg(\frac{\pi y}{2c}))=\frac{2 c arctg(tg(\frac{\pi
x}{2c})+tg(\frac{\pi y}{2c}))}{\pi}=
$$
and
$$
\rho_f(x,y)=\rho(f^{-1}(x),f^{-1}(y))=|tg(\frac{\pi
x}{2c})-tg(\frac{\pi y}{2c})|.
$$
for $x,y \in (-c,c)$.

Then we get a new example of Haar measure space $(~{R}_f,
+_f, \rho_f, \lambda_f )$.
Note that the Haar measure $\lambda_{f}$ in $(-c,c)$ is defined
by
$$
(\forall Y)(Y \in \mathcal{B}_f((-c,c)) \rightarrow
\lambda_{f}(Y)=\lambda(f^{-1}(Y)))=\lambda(\{tg(\frac{\pi w}{2c}):
w \in Y\})).
$$
\end{ex}

\begin{ex} Let
$f: ~{R} \to ~{Z} \times
\{0,1,\cdots,9\}^{~{N}}$ be defined by $f(a_0 + 0,a_1a_2
\cdots)=(a_0,a_1,a_2,\cdots)$ for $a_0 \in ~{Z}$ and
$(a_0,a_1,a_2,\cdots)\in \{0,1,\cdots,9\}^{~{N}}$.

 Then
$f^{-1}:~{Z} \times \{0,1,\cdots,9\}^{~{N}}\to
~{R}$ is defined by $f^{-1}((a_0,a_1,a_2,\cdots))=a_0 +
0,a_1a_2 \cdots$. We put
$$
(a_0,a_1,a_2,\cdots) +_f (b_0,b_1,b_2,\cdots) =
f(f^{-1}((a_0,a_1,a_2,\cdots))+f^{-1}((b_0,b_1,b_2,\cdots)))= $$
$$
f(a_0,a_1a_2 \cdots +b_0,b_1b_2 \cdots)=f(c_0,c_1c_2
\cdots)=(c_0,c_1,c_2 \cdots),
$$
where $c_0,c_1c_2
\cdots=a_0,a_1a_2 \cdots +b_0,b_1b_2 \cdots$.
The metric $ \rho_f$ in $~{Z} \times
\{0,1,\cdots,9\}^{~{N}}$ is defined by

$$
\rho_f((a_0,a_1,a_2,\cdots),(b_0,b_1,b_2,\cdots))=\rho(f^{-1}((a_0,a_1,a_2,\cdots)),f^{-1}((b_0,b_1,b_2,\cdots)))=
$$
$$
\rho(a_0 + 0,a_1a_2 \cdots,b_0 + 0,b_1b_2 \cdots)=|(a_0 +
0,a_1a_2 \cdots)-(b_0 + 0,b_1b_2 \cdots)|.
$$
By Theorem 2.1  we know that $\lambda_{f}$ defined by
$$
(\forall Y)(Y \in \mathcal{B}_{\rho_f}(~{Z} \times
\{0,1,\cdots,9\}^{~{N}}) \rightarrow
\lambda_{f}(Y)=\lambda(f^{-1}Y)= $$
$$\lambda (\{ a_0,a_1a_2\cdots : (a_0,a_1,a_2,\cdots) \in Y )
$$
is Haar measure in $~{Z} \times
\{0,1,\cdots,9\}^{~{N}}$, where $\lambda$ denotes a linear
Lebesgue measure in $~{R}$.
\end{ex}

\begin{rem} Let $M$ be a topological space.
A homeomorphism $\phi : U \to V$ of an open set $U \subseteq M$
onto an open set $V \subseteq ~{R}^d$ will be called a local
coordinate chart (or just `a chart') and $U$ is then a coordinate
neighbourhood (or `a coordinate patch') in $M$.

A  $C^{\infty}$ differentiable structure, or smooth structure, on
$M$ is a collection of coordinate charts $\phi_{\alpha}:
U_{\alpha} \to V_{\alpha} \subseteq ~{R}^d$ (same $d$ for
all $\alpha$'s) such that

(i) $M=\cup_{\alpha \in A}U_{\alpha};$

(ii) any two charts are `compatible': for every $\alpha,\beta$ the
change of local coordinates $\phi_{\beta}\circ \phi_{\alpha}^{-1}$
is a smooth $C^{\infty}$ map on its domain of definition, i.e. on
$\phi_{\alpha}(U_{\beta} \cap U_{\alpha}) \subseteq ~{R}^d$;

(iii)  the collection of charts $\phi_{\alpha}$ is maximal with
respect to the property (ii): if a chart $\phi$ of $M$ is
compatible with all $\phi_{\alpha}$  then $\phi$ is included in
the collection.

A topological space equipped with a $C^{\infty}$ differential
structure is called a real smooth manifold. Then $d$ is called the
dimension of $M$, $d = dim M$.

Recall, that a Lie group is a set $G$ with two structures: $G$ is
a group and $G$ is a real smooth manifold. These structures agree
in the following sense: multiplication and inversion are smooth
maps.

  In \cite{Maleki2012}(see, Example 9.7, p. 64), it is shown that
$G= ~{R}^k$ with $k=\frac{n^2-n}{2}$ has two different Lie
group structure and the Lebesgue measure in $~{R}^k$ is Haar
measure on both Lie groups. Further the author asks(see
,\cite{Maleki2012}, Question 9.8) what are real $k$ dimensional
manifolds with at least two different Lie group structures that
have the same Haar measure.
\end{rem}

The next example answers positively to Maleki's question described in Remark 2.2.

\begin{ex}For $n>2$, let $(~{R}^n, \rho_n, +, \lambda_n)$ be an $n$-dimensional Euclidean vector space  equipped with standard metric $\rho_n$ and $n$-dimensional Lebesgue measure $\lambda_n$.
Let $f : ~{R}^n \to ~{R}^n$ be defined by $f(x_1,x_2,
x_3,\cdots, x_n) = (x_1,x_1^2+x_2,x_3,\cdots, x_n)$ for $(x_1,x_2,
x_3,\cdots, x_n)\in ~{R}^n$.

It is obvious that

1) $f$ is bijection of $R^n$  and  $f^{-1}((x_1,x_2,
x_3,\cdots, x_n))= (x_1,x_2-x_1^2,x_3,\cdots, x_n)$ for $(x_1,x_2,
x_3,\cdots, x_n)\in ~{R}^n$;

2) $f$ as well $f^{-1}$ is infinitely many times continuously differentiable;

3) $f$ is not linear;

4) $f$ as well $f^{-1}$  preserves Lebesgue measure $\lambda_n$.

Let consider $((~{R}^n)_f, (\rho_n)_f, +_f, (\lambda_n)_f)$.
By virtue of Theorem 2.1 we deduce that $((~{R}^n)_f,
(\rho_n)_f, +_f, (\lambda_n)_f)$ is a locally compact non-compact Polish group with two-sided invariant Haar measure $(\lambda_n)_f$.

Note that  $(~{R}^n)_f=~{R}^n$;

b) ~$(\rho_n)_f(x,y)=\rho_n(f^{-1}(x),f^{-1}(y))$;

c) ~$x +_f y=f(f^{-1}(x)+f^{-1}(y))$;

Note that the operation $''+_f''$ is commutative but it  differs
from the usual addition operation$''+''$. Indeed, we have
$$
(1,1, \cdots, 1)+_f(2,2, \cdots, 2)= f(f^{-1}(1,1, \cdots,
1)+f^{-1}(2,2, \cdots, 2))= $$
$$f((1,0,1, \cdots, 1)+(2,-2,2,
\cdots, 2))=f(3,-2,3,\cdots,3)=(3,7,3,\cdots,3)$$

and

$$
(1, \cdots, 1)+(2, \cdots, 2)=(3,\cdots,3).
$$

Since $f$ is Borel measurable, by using Theorem 2.1  we deduce
that $\mathcal{B}_{\rho_f}(~{R}^n)=\mathcal{B}(~{R}^n)$.

Note also that $(\lambda_n)_f=\lambda_n$. Indeed, by
Theorem 2.1 we have  that
$$
(\forall Y)(Y \in \mathcal{B}(~{R}^n) \rightarrow
\lambda_{f}(Y)=\lambda(f^{-1}(Y))= \lambda(Y)).
$$
\end{ex}

\begin{rem}Notice that Example 2.5 extends the result of Example 9.7
\cite{Maleki2012}. Indeed, it is obvious that for $n>2,$ measure
space $((~{R}^n)_f, (\rho_n)_f, +_f,
(\lambda_n)_f)=(~{R}^n, (\rho_n)_f, +_f, \lambda_n)$  has
Lie group structure which differs from standard  Lie group structure
of $~{R}^n$  because group operations $''+''$ and $''+_f''$, as were showed in Example 2.8, are different.  Furthermore the Lebesgue measure $\lambda_n$ (in $R^n$) is Haar measure  on both Lie groups.

\end{rem}

\medskip

Now let consider $\ell_2=\{(x_k)_{k \in ~{N}}: x_k \in
~{R}~\&~ k \in ~{N}~\&~ \sum_{k \in
~{N}}x_k^2<\infty\}$ as a vector space with usual addition
operation $" + "$. If we equip $\ell_2$ with standard metric
$\rho_{\ell_2}$ defined by
$$
\rho_{\ell_2}((x_k)_{k \in ~{N}},(y_k)_{k \in
~{N}})=\sqrt{\sum_{k \in ~{N}}(x_k-y_k)^2}
$$
for $(x_k)_{k \in ~{N}},(y_k)_{k \in ~{N}} \in
\ell_2$, then $(\ell_2,'' + '', \rho_{\ell_2})$ stands an
example of a non-locally compact Polish group. Here naturally
arise a question asking whether there exists a metric $\rho$ in
$\ell_2$ such that $\big(\ell_2," + ", \rho\big)$ stands an
example of a locally compact $\sigma$-compact Polish group. An
affirmative answer to this question is containing in the following
example.

\begin{ex} Let consider $~{R}$ and $\ell_2$ as vector spaces
over the group of all rational numbers $~{Q}$. Let $(a_i)_{i
\in I}$ and $(b_i)_{i \in I}$ be Hamel bases in $~{R}$ and
$\ell_2$, respectively. For $x \in ~{R}\setminus \{0\}$,
there exists a unique sequence of non-zero rational numbers
$(q^{(x)}_{i_k})_{1 \le k \le n_x}$ such that
$x=\sum_{k=1}^{n_x}q^{(x)}_{i_k}a_{i_k}$. We set
$f(x)=\sum_{k=1}^{n_x}q^{(x)}_{i_k}b_{i_k}$ for $x \in
~{R}\setminus \{0\}$ and $f(0)=(0,0,\cdots)$. Notice that
$f:~{R}\to\ell_2$ is one-to-one linear transformation.

Let $x=\sum_{k=1}^{n_x}q^{(x)}_{i_k}b_{i_k}$ and
$y=\sum_{k=1}^{n_y}q^{(y)}_{i_k}b_{i_k}$. Now  if we set
$$
x +_f y = f(f^{-1}(x)\odot f^{-1}(y)),
$$
then we will obtain
$$
x +_f y = f(f^{-1}(x)+f^{-1}(y))=
f(\sum_{k=1}^{n_x}q^{(x)}_{i_k}a_{i_k}+\sum_{k=1}^{n_y}q^{(y)}_{i_k}a_{i_k})=$$
$$
\sum_{k=1}^{n_x}q^{(x)}_{i_k}b_{i_k}+\sum_{k=1}^{n_y}q^{(y)}_{i_k}b_{i_k}=
x+y,
$$
which means that a group operation $+_f$ coincides with usual
addition operation $" + "$.

Let define $\rho$ by
$$
\rho(x,y)=|f^{-1}(x)-f^{-1}(y)|=|\sum_{k=1}n_xq^{(x)}_{i_k}a_{i_k}
-  \sum_{k=1}n_yq^{(y)}_{i_k}a_{i_k}|.
$$

By Theorem 2.1 we know that $(~{R}_f, +_f, \rho_f)$,
equivalently $(\ell_2, + , \rho_f)$ is a locally compact non-compact Polish group which is
isomorphic to the Polish group $(~{R}, +, |\cdot|)$.

Moreover, if $(~{R}, +, |\cdot|, \lambda)$ is Haar measure
space, then $(\ell_2, + , \rho_f, \lambda_f)$ also is Haar measure
space. Denoting by $\mathcal{B}_{\rho_f}(\ell_2)$ a Borel
$\sigma$-algebra of subsets of $\ell_2$ generated by the metric
$\rho_f$, we define Haar measure $\lambda_{f}$ in $\ell_2$ by
$$
(\forall Y)(Y \in \mathcal{B}_{\rho_f}(\ell_2) \rightarrow
\lambda_{f}(Y)=\lambda(f^{-1}(Y))).
$$
\end{ex}

\begin{rem} Let $(G,\rho, +)$ be an abelian Polish group.  We say that $G$ is
one-dimensional group w.r.t. metric $\rho$ if for each $n \in
~{N}$ and for each family of different elements $(a_k)_{1 \le k \le
n}$ there is permutation $h$ of $\{1,2, \cdots, n\}$ such that
$$
\rho(a_{h(1)},a_{h(n)})=\sum_{k=1}^{n-1}\rho(a_{h(k)},a_{h(k+1)}).
$$
Then it is obvious to show that $(\ell_2, + , \rho_f, \lambda_f)$
is one-dimensional group w.r.t. metric $\rho_f$.
\end{rem}

\begin{ex} Let consider $~{R}^{\infty}$ and $~{R}$  as vector spaces
over the group of all rational numbers $~{Q}$. Let $(a_i)_{i
\in I}$ and $(b_i)_{i \in I}$ be  Hamel bases in $~{R}^{\infty}$ and
$~{R}$, respectively. For $x \in ~{R}^{\infty} \setminus \{(0,0, \cdots)\}$,
there exists a unique sequence of non-zero rational numbers
$(q^{(x)}_{i_k})_{1 \le k \le n_x}$ such that
$x=\sum_{k=1}^{n_x}q^{(x)}_{i_k}a_{i_k}$. We set
$f(x)=\sum_{k=1}^{n_x}q^{(x)}_{i_k}b_{i_k}$ for $x \in
~{R}^{\infty} \setminus \{(0,0, \cdots)\}$ and $f(0,0,\cdots)=0$. Notice that
$f:~{R}^{\infty} \to ~{R}$ is one-to-one linear transformation.

For $w,z \in ~{R}$, setting
$$
w +_f z = f(f^{-1}(w) + f^{-1}(z)),
$$
we get
$$
w +_f z = f(f^{-1}(w)+f^{-1}(z))=
f(\sum_{k=1}^{n_w}q^{(w)}_{i_k}a_{i_k}+\sum_{k=1}^{n_z}q^{(z)}_{i_k}a_{i_k})=$$
$$
\sum_{k=1}^{n_w}q^{(w)}_{i_k}b_{i_k}+\sum_{k=1}^{n_z}q^{(z)}_{i_k}b_{i_k}=
w+z,
$$
which means that a group operation $+_f$ coincides with usual
addition operation $'' + ''$  in $~{R}$.

Let define $\rho$ by
$$
\rho(w,z)=\rho_T(f^{-1}(w),f^{-1}(z)),$$
where $\rho_T$ is Tychonov  metric in  $~{R}^{\infty}$ defined by
$$
\rho_T((x_k)_{k \in N},(y_k)_{k \in N})=\sum_{k=1}^{\infty}\frac{|x_k-y_k|}{2^k(1+|x_k-y_k|)}
$$
for $(x_k)_{k \in N},(y_k)_{k \in N} \in ~{R}^{\infty}$.

By Theorem 2.1 we know that $(~{R}^{\infty}_f, +_f, \rho_f)$,
equivalently, $(~{R}, + , \rho_f)$ is an abelian  non-locally compact Polish group which is
isomorphic to the abelian  non-locally compact Polish group $(~{R}^{\infty}, +, \rho_T)$.

Let $\lambda$ be a translation invariant quasifinite borel measure in $~{R}^{\infty}$(see, for example, \cite{Bak91}, \cite{Bak04}).

We put
$$
(\forall Y)(Y \in \mathcal{B}_{\rho_f}(~{R}) \rightarrow
\lambda_{f}(Y)=\lambda(f^{-1}(Y))).
$$
Since   $\lambda$ is translation invariant quasifinite borel measure in $~{R}^{\infty}$, by virtue of Theorem 2.1 we deduce that so is
the measure $\lambda_{f}$ in $(~{R}, + , \rho_f)$.
\end{ex}

\section{Is an arbitrary diffused Borel probability measure in a Polish space Haar measure?}

The following lemma is a useful ingredient for our further investigations.
\begin{lem}
Let $E_1$ and $E_2$ be any two Polish topological spaces  without isolated points. Let
$\mu_1$ be a probability diffused Borel measure on $E_1$ and let
$\mu_2$ be a probability diffused Borel measure on $E_2$. Then
there exists a Borel isomorphism $\varphi : (E_1,B(E_1)) \to
(E_2,B(E_2))$ such that
$$\mu_1(X)=\mu_2(\varphi(X))$$ for every $ X \in B(E_1)$.
\end{lem}

The proof of Lemma 3.1 can be found in \cite{CicKha95}.

The solution of the Problem 1.2 is contained in the following statement.

\begin{thm} Let $(G,\rho)$ be a Polish metric space which is dense-in-itself. Let $\mu$ be a diffused Borel probability measure defined in $(G,\rho)$. Then there exist a metric $\rho_{\varphi}$ and a group operation $\odot_{\varphi}$ in $G$ such that the following three conditions

(i) The class of Borel measurable subsets of $G$ generated by the metric $\rho_{\varphi}$ coincides with the class of Borel
measurable subsets of the same space generated by the metric $\rho$,

(ii) $(G,\rho_{\varphi}, \odot_{\varphi})$  is  a compact Polish group

and

(iii) $\mu$ is a left (right or two-sided) invariant  probability Haar  measure in $(G,\rho_{\varphi}, \odot_{\varphi})$

hold true.

\end{thm}

\begin{proof} Let $(G_2,\rho_2,\odot_2)$ be a compact Polish group which is dense-in-itself equipped with two-sided invariant Haar
measure $\lambda_2$. By Lemma 3.1, there
exists a Borel isomorphism $\varphi : (G,B(G)) \to (G_2,B(G_2))$
such that
$$\mu(X)=\lambda_2(\varphi(X))$$ for every $ X \in B(G)$.

We set
$$x \odot_{\varphi} y=\varphi^{-1}( \varphi(x)\odot_2
\varphi(y))
$$
and
$$
\rho_{\varphi}(x,y)=\rho_2(\varphi(x),\varphi(y))
$$
for $x,y \in G$.

By Theorem 2.1 we know that $(G,\odot_{\varphi}, \rho_{\varphi})$  is a
compact Polish group without isolated points which is Borel isomorphic to the compact Polish group $(G_2,
\odot_2, \rho_2)$ and a measure $\lambda_{\varphi}$,
defined by
$$
(\forall Y)(Y \in \mathcal{B}(G_2) \rightarrow
\lambda_{\varphi}(Y)=\lambda(\varphi^{-1}(Y))),
$$
is a two-sided invariant Haar measure in $G$.

Since $\varphi : (G,B(G)) \to (G_2,B(G_2))$ is Borel isomorphism,
we deduce that
$$
\{ z : \rho_{\varphi}(x,z)<r \}=\{ z :
\rho_2(\varphi(x),\varphi(z))<r \}=\varphi^{-1}(\{w :
\rho_2(\varphi(x),w)<r  \}) \in B(G_2).
$$
for each $x \in G$ and $r>0$.

Since $\mathcal{B}(G)$ is $\sigma$-algebra, we deduce that
$\mathcal{B}_{\rho_{\varphi}}(G) \subseteq \mathcal{B}(G)$.

We have to show that $\mathcal{B}(G) \subseteq
\mathcal{B}_{\rho_{\varphi}}(G) $. Assume the contrary and let $X \in
\mathcal{B}(G) \setminus \mathcal{B}_{\rho_{\varphi}}(G)$. Since $\varphi
: (G,\mathcal{B}(G)) \to (G_2,\mathcal{B}(G_2))$ is Borel isomorphism, we deduce
$\varphi(X) \in \mathcal{B}(G_2)$. Then, by Theorem 2.1 we deduce  that  $X \in
B_{\rho_{\varphi}}(G)$ and we get the contradiction.

\end{proof}

\begin{rem} In the proof of Theorem 3.1, if under $(G_2,\rho_2,\odot_2)$ we take an abelian compact Polish group without isolated points and with a two-sided invariant Haar
measure $\lambda$ then the group  $(G,\rho_{\varphi}, \odot_{\varphi})$  will be  a compact abelian Polish group without isolated points. Similarly, if under $(G_2,\rho_2,\odot_2)$ we take a non-abelian compact Polish group without isolated points and with a two-sided invariant Haar
measure $\lambda$ then the  group  $(G,\rho_{\varphi}, \odot_{\varphi})$  also will be  a non-abelian compact Polish group without isolated points.

\end{rem}

The solution of Problem 1.3 is contained in the  following statement.

\begin{thm} Let $(G,\rho)$ be a Polish metric space which is dense-in-itself. Let $\mu$ be a diffused $\sigma$-finite non-finite Borel measure defined in $(G,\rho)$. Then there exist a metric $\rho_{\varphi}$, a group operation $\odot_{\varphi}$ in $G$ and the Borel measure $\mu^\star$ in $G$ such that the following  conditions

(i) The class of Borel measurable subsets of $G$ generated by the metric $\rho_{\varphi}$ coincides with the class of Borel
measurable subsets of the same space generated by the metric $\rho$,

(ii) $(G,\rho_{\varphi}, \odot_{\varphi})$  is  a non-compact locally compact Polish group,

(iii) The measures $\mu^\star$ and $\mu$ are equivalent,

and

(iv) $\mu^\star$ is a left (right or two-sided) invariant  $\sigma$-finite non-finite  Haar  measure in $(G,\rho_{\varphi}, \odot_{\varphi})$

hold true.

\end{thm}

\begin{proof} Let $(G_2,\rho_2,\odot_2)$ be a non-compact locally compact Polish group which is dense-in-itself with two-sided invariant  $\sigma$-finite non-finite  Haar  measure  $\lambda_2$ (for example, the real axis  $\mathbf{R}$  with Lebesgue measure ).
Let $(X^{(2)}_k)_{k \in N}$ be a partition of the $G_2$ into Borel measurable subsets such that $0< \lambda_2(X^{(2)}_k)<+\infty$ for $k \in N$.
We set $$\mu_2(X)=\sum_{k \in N}\frac{\lambda_2(X \cap X^{(2)}_k)}{2^k\lambda_2(X^{(2)}_k)}$$ for $X \in \mathcal{B}(G_2)$.

Similarly, let $(Y_k)_{k \in N}$ be a partition of the $G$ into Borel measurable subsets such that $0< \mu(Y_k)<+\infty$ for $k \in N$.
We set $$\mu_1(Y)=\sum_{k \in N}\frac{\mu(Y \cap Y_k)}{2^k\mu(Y_k)}$$ for $Y \in \mathcal{B}(G)$.

By Lemma 3.1, there
exists a Borel isomorphism $\varphi : (G,B(G)) \to (G_2,B(G_2))$
such that
$$\mu_1(Y)=\mu_2(\varphi(Y))$$ for every $ Y \in B(G)$.

We set
$$x \odot_{\varphi} y=\varphi^{-1}( \varphi(x)\odot_2
\varphi(y))
$$
and
$$
\rho_{\varphi}(x,y)=\rho_2(\varphi(x),\varphi(y))
$$
for $x,y \in G$.

By Theorem 2.1 we know that $(G,\odot_{\varphi}, \rho_{\varphi})$  is a locally
compact non-compact Polish group without isolated points which is Borel isomorphic to the non-compact locally compact Polish group $(G_2,
\odot_2, \rho_2)$.

Now we put

$$
\mu^\star(X)=\sum_{k \in N}2^k\lambda_2(X^{(2)}_k) \mu_1(X \cap \varphi^{-1}(X^{(2)}_k))$$
for $X \in \mathcal{B}(G)$.

By using Theorem 2.1 and the coincidence of Borel $\sigma$-algebras $\mathcal{B}(G)$ and $\mathcal{B}_{\rho_{\varphi}}(G)$,  we have to show only that the measure $\mu^\star$ is a two-sided invariant measure in $G$.
Indeed, for $h_1,h_2 \in G$ and $X \in \mathcal{B}(G)$, we have
$$
\mu^\star(h_1 \odot_{\varphi} X  \odot_{\varphi} h_2 )=\sum_{k \in N}2^k\lambda_2(X^{(2)}_k) \mu_1((h_1 \odot_{\varphi} X  \odot_{\varphi} h_2) \cap \varphi^{-1}(X^{(2)}_k))=
$$

$$
\sum_{k \in N}2^k\lambda_2(X^{(2)}_k) \mu_2(\varphi[(h_1 \odot_{\varphi} X  \odot_{\varphi} h_2) \cap \varphi^{-1}(X^{(2)}_k])=
$$

$$
\sum_{k \in N}2^k\lambda_2(X^{(2)}_k) \sum_{i \in N}\frac{\lambda_2(\varphi[(h_1 \odot_{\varphi} X  \odot_{\varphi} h_2) \cap \varphi^{-1}(X^{(2)}_k)] \cap X^{(2)}_i)}{2^i\lambda_2(X^{(2)}_i)}=
$$

$$
\sum_{k \in N}2^k\lambda_2(X^{(2)}_k) \sum_{i \in N}\frac{\lambda_2(\varphi[\varphi^{-1}\{ \varphi h_1 \odot \varphi(X)  \odot  \varphi(h_2)\} \cap \varphi^{-1}(X^{(2)}_k)] \cap X^{(2)}_i)}{2^i\lambda_2(X^{(2)}_i)}=
$$

$$
\sum_{k \in N}2^k\lambda_2(X^{(2)}_k) \sum_{i \in N}\frac{\lambda_2(( (\varphi h_1 \odot \varphi(X)  \odot  \varphi(h_2)) \cap X^{(2)}_k) \cap X^{(2)}_i)}{2^i\lambda_2(X^{(2)}_i)}=
$$

$$
\sum_{k \in N} \lambda_2(( \varphi h_1 \odot \varphi(X)  \odot  \varphi(h_2)) \cap X^{(2)}_k)=
$$

$$
\lambda_2(\varphi h_1 \odot \varphi(X)  \odot  \varphi(h_2))=\lambda_2( \varphi(X))=
\sum_{k \in N}2^k\lambda_2(X^{(2)}_k) \mu_2(\varphi(X) \cap X^{(2)}_k)=
$$
$$
\sum_{k \in N}2^k\lambda_2(X^{(2)}_k) \mu_1(\varphi^{-1}[\varphi(X) \cap X^{(2)}_k])=
$$
$$
\sum_{k \in N}2^k\lambda_2(X^{(2)}_k) \mu_1(X\cap \varphi^{-1}(X^{(2)}_k))=\mu^\star( X ).
$$
\end{proof}

\begin{rem} The result of Theorem 3.2 remains true if $\mu$ is a diffused Borel probability measure in $(G,\rho)$.
\end{rem}

As a simple consequence of Theorem 3.2, we have the following corollary.

\begin{cor} Let $(G,\rho)$ be a Polish metric space which is dense-in-itself. Let $\mu$ be a diffused $\sigma$-finite non-finite Borel measure defined in $(G,\rho)$. Then there exist a metric $\rho_{\varphi}$  and  a group operation $\odot_{\varphi}$ in $G$ such that  the following three  conditions

(i) The class of Borel measurable subsets of $G$ generated by the metric $\rho_{\varphi}$ coincides with the class of Borel
measurable subsets of the same space generated by the metric $\rho$,

(ii) $(G,\rho_{\varphi}, \odot_{\varphi})$  is  a non-compact locally compact Polish group

and

(iii) The measure $\mu$ is a two-sided quasi-invariant \footnote{A Borel measure $\mu$ defined in a Polish group $(G, \odot, \rho)$ is called two-sided quasi-invariant measure in $G$ if for each Borel subset $X$ we have  $\mu(X)>0$ if and only  $\mu(h_1 \odot X \odot h_2)>0$ for each pair of  elements $h_1,h_2 \in G$.}  Borel probability  measure in $(G,\rho_{\varphi}, \odot_{\varphi})$

hold true.

\end{cor}

Finally, we state the following problem

{\bf Problem 3.1}
 Let $(G,\rho)$ be a Polish metric space which is dense-in-itself, that is, $G$ is a space homeomorphic to a separable complete metric
space and  $G$ has no isolated points. Let $\mu$ be a diffused non-finite $\sigma$-finite Borel measure defined in $(G,\rho)$. Do there exist a metric $\rho_1$ and a
group operation $\odot$ in $G$ such that the following three conditions

(j) The class of Borel measurable subsets of $G$ generated by the metric $\rho_1$ coincides with the class of Borel
measurable subsets of the same space generated by the metric $\rho$,

(jj) $(G,\rho_1, \odot)$ is a non-compact locally compact  Polish group

and

(jjj) $\mu$ is a left(right or two-sided) invariant non-finite $\sigma$-finite Haar  measure in $(G,\rho_1, \odot)$

hold true ?

{\bf Acknowledgement}. The main results of this manuscript were reported on Swedish-Georgian Conference in Analysis~ $\&$ ~Dynamical Systems which was held   at National Academy of Georgian Republic in Tbilisi, Georgia, 15-22 July, 2015.

\end{document}